\documentclass{elsarticle}
\usepackage{amsthm,amssymb,amsmath, latexsym}
\usepackage[T1]{fontenc} 
\usepackage{times,url,xcolor}
  
\newtheorem{theorem}{Theorem}
\newtheorem{lemma}[theorem]{Lemma}

\newtheorem{corollary}[theorem]{Corollary}
\newdefinition{remark}[theorem]{Remark}
\newdefinition{example}[theorem]{Example}
\newcommand{\myem}[1]{{\textbf{\textit{#1}}}}
\newcommand{\grp}[1]{\langle #1\rangle}  

\allowdisplaybreaks

\def\sref#1{\S$\ref{#1}$}
\def\lref#1{Lemma~$\ref{#1}$}
\def\cyref#1{Corollary~$\ref{#1}$}
\def\tref#1{Theorem~$\ref{#1}$}
\def\Tref#1{Table~$\ref{#1}$}
\def\fref#1{Figure~$\ref{#1}$}
\def\egref#1{Example~$\ref{#1}$}

\renewcommand{\geq}{\geqslant}
\renewcommand{\leq}{\leqslant}
\renewcommand{\ge}{\geqslant}
\renewcommand{\le}{\leqslant}

\newcommand{\shape}{\mathcal{S}}
\newcommand{\trip}{\mathcal{T}}

\allowdisplaybreaks

\makeatletter
\def\ps@pprintTitle{%
 \let\@oddhead\@empty
 \let\@evenhead\@empty
 \def\@oddfoot{\centerline{\thepage}}%
 \let\@evenfoot\@oddfoot}
\makeatother

\begin{document}

\title{Small partial Latin squares that embed in an\\
infinite group but not into any finite group\tnoteref{t1}}
\tnotetext[t1]{This research was supported by the ARC grant DP150100506. We thank Alexander Hulpke for some GAP advice, in particular, for pointing us to the GAP implementation of his paper \cite{AH}. }

\author{Heiko Dietrich\fnref{cor1}}
\ead{heiko.dietrich@monash.edu}
\author{Ian M.\ Wanless}
\ead{ian.wanless@monash.edu}
 
\fntext[cor1]{Corresponding author}
\address{School of Mathematical Sciences, Monash University, VIC 3800, Australia}


\begin{abstract}
Suppose that $Y_1,Y_2,Y_3$ are finite sets and
$P\subseteq Y_1\times Y_2\times Y_3$. We say that $P$ embeds in a group
$G$ if there exist injective maps $\phi_i\colon Y_i\rightarrow G$ for $i=1,2,3$
such that $\phi_1(y_1)\phi_2(y_2)=\phi_3(y_3)$ for each
$(y_1,y_2,y_3)\in P$. Hirsch and Jackson asked for the cardinality of the
smallest $P$ that embeds in some infinite group but not into any
finite group. We prove that the answer to their question is 12.
Moreover, we show that there are 50 examples of cardinality 12, up to 
equivalence, and
each of them embeds in the (infinite) Baumslag group $G=\langle a,b \mid
b=[b,b^a]\rangle$. Our proof uses computations to answer
questions about finitely presented groups which are known to be
algorithmically undecidable in general.
\end{abstract}

\maketitle

\vspace*{-0.586cm}

\section{Introduction and results} 

We define a \myem{partial Latin square} (PLS) $P$ to be a matrix in
which some cells may be empty and in which each filled cell contains
one symbol from an underlying alphabet $\Lambda(P)$,
such that no symbol occurs more than once within any row or
column. The \myem{size} of $P$ is the number of filled cells in $P$,
and the \myem{shape} $\shape(P)$ of $P$ is the set of filled cells in
$P$.  We avoid degeneracies by insisting that each row and column
contains at least one filled cell, and that each element of
$\Lambda(P)$ appears at least once in $P$. A side-effect is that our
PLS need not be square matrices. In some references, PLS are defined
to be square matrices and to have at least as many rows as there are
symbols. To achieve these properties it is always possible to add
empty rows and/or empty columns to our PLS.  Allowing any finite
number of empty rows and columns would create some technical nuisances
but would not materially affect any of the questions we are interested in
solving.

Let $P=[P_{i,j}]$ be an $m\times n$ PLS.
An \myem{embedding} $\phi\colon P\hookrightarrow G$ of $P$ into a group $G$ 
is a triple $(\phi_1,\phi_2,\phi_3)$ of injective maps
\[ \phi_1\colon \{1,2,\ldots,m\}\to G,\quad \phi_2\colon \{1,2,\ldots,n\}\to G,\quad \phi_3\colon \Lambda(P)\to G,\]
such that $\phi_1(i)\phi_2(j)=\phi_3(P_{i,j})$ for all $(i,j)\in\shape(P)$. 
We refer to $r_i=\phi_1(i)$ for $i=1,\dots,m$ as the \myem{row labels},
and $c_j=\phi_2(j)$ for $j=1,\dots,n$ as the \myem{column labels}.
Intuitively, $P$ embeds in $G$ if a copy of $P$ can be found within
the Cayley table of $G$ (in the subtable which has row
and column labels $r_1,\ldots,r_m$ and $c_1,\ldots,c_n$,
respectively). We refer to \cite{CW09} for a discussion of applications of
embeddings of PLS in groups and connections with linear algebra
and topological graph theory. Note that embedding in groups is a special type
of completion problem for PLS. See \cite[Chap.3]{DKIII}
for an introduction to the rich literature of such problems.

A useful notion of equivalence for PLS is obtained by converting each
PLS $P$ into a set of triples $\trip(P)=\{(r,c,P_{r,c}):(r,c)\in\shape(P)\}$.
We say that two PLS are from the same \myem{species} if they produce
the same set of triples, modulo uniform permutation of the 3
coordinates in the triples, and relabelling within each coordinate.
The property of having an embedding in a given group is 
a species invariant \cite[Lem.1]{CW09}, so in this paper it will
suffice to consider one representative from each species of PLS.

Hirsch and Jackson \cite[Ex.3.7]{HJ12} gave an example of a PLS of
size 29 that can be embedded in an infinite group, but not in any
finite group.  They noted that smaller examples exist, and posed the
question of what the smallest possible size of such a PLS is. The
purpose of this paper is to answer this question by proving:

\begin{theorem}\label{t:main}
There are $50$ species of PLS of size $12$ that
can be embedded in an infinite group, but not in any finite group.
No  PLS of smaller size has the same property. 
\end{theorem}

One of the 50 PLS of size $12$ is detailed in \lref{lemBS2}. The
remainder of the paper discusses our computational approach
which proves \tref{t:main}. These computations involve manipulating finite
presentations for groups to test, for example, 
whether a group is finite or whether two group elements are equal. 
These questions are equivalent to solving the ``\emph{word problem}'' and 
hence, in full generality, are well known to be
algorithmically undecidable, cf.~\cite[\S 5]{handbook} or \cite[p.54]{Rob82}. 
So we were somewhat fortunate to find methods which 
solved the instances of these problems that we needed to solve.

For recent work related to this paper we refer to \cite{WW17}, which
identifies all smallest PLS that (a)~do not embed into any group,
(b)~embed into a group but do not embed into any abelian group, or
(c)~embed into an abelian group but do not embed into any cyclic
group. Each of those classes contains small
PLS that can be found with simpler methods than
we employ in the present paper.

\section{The presentation defined by a PLS}

The next lemma (cf.~\cite[Lem.2]{CW09}) is an easy observation and shows that we lose no generality by assuming that the associated row and column labels of an embedding satisfy $r_1=c_1=1$, the identity of the group.
 
\begin{lemma}\label{lemrc1}
If a PLS $P$ can be embedded into
a group $G$, then it can be embedded
with row labels $r_1,\ldots,r_m\in G$ and column labels $c_1,\ldots,c_n\in G$ 
that satisfy $r_1=c_1=1$, the identity in $G$.
\end{lemma}

\begin{proof}
If  $(\phi_1,\phi_2,\phi_3)\colon P\hookrightarrow G$ is an embedding
with associated row labels
$r_1,\ldots,r_m$ and column labels $c_1,\ldots,c_n$, then $(\phi'_1,\phi'_2,\phi'_3)\colon P\hookrightarrow G$ defined by $\phi_1'(r)=r_1^{-1}\phi_1(r)$, $\phi_2'(c)=\phi_2(c)c_1^{-1}$, and $\phi_3'(e)=r_1^{-1}\phi_3(e)c_1^{-1}$ is an embedding with the required property. 
\end{proof}

Let $P$ be an $m\times n$ PLS. Let $F$ be the free group on $X=\{R_i,C_j: (i,j)\in\shape(P)\}\cup\Lambda(P)$ and let \[\mathcal{R}=\{R_iC_j{P}_{i,j}^{-1}\in F: (i,j)\in\shape(P)\}.\] The \myem{presentation defined by $P$} is $\{X\mid\mathcal{R}\}$ and the \myem{group defined by $P$} is $\grp{P}=\langle X\mid \mathcal{R}\rangle$.

\begin{lemma}\label{lemQ}
Let $P$ be an $m\times n$ PLS, and let $(\phi_1,\phi_2,\phi_3)\colon P\hookrightarrow G$ be an embedding into a group $G$, with associated row and column labels $r_1,\ldots,r_m$ and $c_1,\ldots,c_n$, respectively. Then $P$ embeds in the subgroup $H$ of $G$ generated by $\{r_i,c_j,\phi_3(P_{i,j}): (i,j)\in\shape(P)\}$, and $H$ is a quotient of $\grp{P}$. In particular, $P$ embeds into $\grp{P}$.
\end{lemma}

\begin{proof}
Clearly, $P$ embeds into the group $H$, and the generators of $H$ satisfy $r_ic_j=\phi_3(P_{i,j})$ for all $(i,j)\in\shape(P)$. By von Dyck's Theorem \cite[2.2.1]{Rob82}, the map $X\to H$ defined by $R_i\mapsto r_i$, $C_j\mapsto c_j$, $P_{i,j}\mapsto \phi_3(P_{i,j})$ for all  $(i,j)\in\shape(P)$ extends to a surjective group homomorphism $\grp{P}\to H$.
\end{proof}

In later sections we describe how we used the computer algebra system
GAP \cite{gap} to answer questions about possible embeddings of
PLS. Typically, \lref{lemQ} provides the starting point for these
investigations, since it says that if $P$ is going to embed in any
group, it must embed in $\grp{P}$. In the remainder of this section we
explore some of consequences of finding an embedding.

Recall that a group $G$ is \emph{residually finite} if for every nontrivial $g\in G$ there is a normal subgroup $N\unlhd G$ with $g\notin N$ such that $G/N$ is finite.

\begin{lemma}\label{lemRF}
Let $P$ be a PLS such that $\grp{P}$ is a residually finite group. If $P$ can be embedded into $\grp{P}$, then $P$ can be embedded into a finite group.
\end{lemma}  

\begin{proof}
Let $\phi=(\phi_1,\phi_2,\phi_3)\colon P\hookrightarrow \grp{P}$ 
be an embedding, with associated row and
column labels $r_1,\ldots,r_m$ and $c_1,\ldots,c_n$, respectively. Let
$\{s_1,\ldots,s_k\}=\{\phi_3(s): s\in\Lambda(P)\}$ be the set
of symbols in the embedded PLS. 
Define 
\[\mathcal{M}=
\{r_ir_j^{-1}:1\le i<j\le m\}\cup
\{c_ic_j^{-1}:1\le i<j\le n\}\cup
\{s_is_j^{-1}:1\le i<j\le k\}.
\]
Note that each $r\in\mathcal{M}$ is nontrivial since $\phi$ is an
embedding.  By assumption,  for each
$r\in\mathcal{M}$ there exists a normal subgroup $U_r<\grp{P}$ of finite
index with $r\notin U_r$. Since the intersection of two finite-index
subgroups has finite index, see \cite[1.3.11]{Rob82}, it follows that
$U=\bigcap_{r\in\mathcal{M}} U_r$ is a normal subgroup of $\grp{P}$ of
finite index, that is, $H=\grp{P}/U$ is finite. By construction, $r\notin
U$ for all $r\in\mathcal{M}$. Thus, if $\pi\colon \grp{P}\to H$ is the
natural epimorphism, then 
$\phi'=(\pi\circ \phi_1,\pi\circ \phi_2,\pi\circ \phi_3)$ 
embeds $P$ into $H$.
\end{proof}

Note that every free group is residually finite, see \cite[6.1.9]{Rob82}; this yields the following corollary.

\begin{corollary}\label{cy:lemRF}
Let $P$ be a PLS such that $\grp{P}$ is a free group. If $P$ can be embedded into $\grp{P}$, then $P$ can be embedded into a finite group.
\end{corollary}

\begin{lemma}\label{lemInfAb}
Let $P$ be a PLS. If $P$ can be embedded into the abelianisation $\grp{P}/\grp{P}'$, then $P$ can be embedded into a finite abelian group.
\end{lemma}
\begin{proof}
Write $G=\grp{P}$ and $A=G/G'$. Let $(\phi_1,\phi_2,\phi_3)\colon P\hookrightarrow A$ be an embedding, with row and column labels $r_1,\ldots,r_u$ and $c_1,\ldots,c_v$, respectively. The fundamental theorem of finitely generated abelian groups \cite[4.2.10]{Rob82} shows that  $A\cong Z_{1}\times\cdots\times Z_{k}\times F$ for some finite group $F$ and integer $k\ge0$, where each $Z_{i}\cong\mathbb{Z}$ is a infinite cyclic group with generator $g_i$. Thus each label $r_i$, $c_j$, and each symbol in $\phi_3(\Lambda(P))$ can be written as  $g_1^{e_1}\cdots g_k^{e_k}f$ for certain $e_1,\ldots,e_k\in\mathbb{Z}$ and $f\in F$.  
Let $\mathcal{E}_i$ be the set of exponents of $g_i$ occurring in all these expressions. Now choose an integer $m_i$ such that $m_i>2|e|$ for all $e\in\mathcal{E}_i$. Then $B=A /\langle g_1^{m_1},\ldots,g_k^{m_k}\rangle$ is a finite abelian group, and, by construction, $P$ can be embedded into $B$.
\end{proof}

\section{An example of size 12}\label{s:eg12}

In this section we explicitly present a PLS of size 12 that cannot be
embedded in any finite group but can be embedded in an infinite group.
This will turn out to be the smallest possible size for such an example.

\begin{lemma}\label{lemBS2}
The PLS
\[
P=\begin{array}{|c|c|c|c|c|}\hline
 & b & & c &   \\\hline
a&   & c & & \\\hline
 & a & b & & \\\hline
b&  & &  &d \\\hline
c&d   & &   &\\\hline
 &   & & b& c\\\hline
  \end{array}
\]
can be embedded in an infinite group, but in no finite group.
\end{lemma}

\begin{proof} 
In \cite{Bau69}, Baumslag defined the finitely presented group  $G=\langle A,B\mid B=[B,B^A]\rangle$, where, as usual, $B^A=A^{-1}BA$ and $[B,B^A]=B^{-1}B^{(B^A)}$. He showed that $G$ is infinite, non-cyclic, and $B=1$ in every finite quotient of $G$. We use these properties to prove the claim.
   
First, suppose we have an embedding $(\phi_1,\phi_2,\phi_3)\colon P\hookrightarrow H$ into some group $H$, with row and column labels $r_1,\ldots,r_6$ and $c_1,\ldots,c_5$, respectively. For simplicity, we identify $a,b,c,d$ with their images under $\phi_3$, and, by \lref{lemrc1}, we assume that $r_1=c_1=1$. We deduce from \fref{figBS2} that 
\begin{equation}\label{eqBS2}
\begin{aligned}
   r_2&=a,     \hspace*{2cm}&   c_2&=b,          \hspace*{2cm} & c&=aba^{-1}b,    \\
   r_3&=ab^{-1},             &   c_3&=ba^{-1}b,                  & d&=aba^{-1}b^2,   \\
   r_4&=b,                  &   c_4&=aba^{-1}b,                      \\         
   r_5&=aba^{-1}b,           &   c_5&=ab^2a^{-1}b,                      \\
   r_6&=ab^{-1}a^{-1}. 
\end{aligned}
\end{equation}
Hence, $bab^2a^{-1}b=r_4c_5=d=aba^{-1}b^2$, that is, $bab^2a^{-1}=aba^{-1}b$, and so $b^2=a^{-1}b^{-1}aba^{-1}ba$ and $b=b^{-1}a^{-1}b^{-1}aba^{-1}ba$, which can be written as $b=[b,b^a]$. This shows that the subgroup $K=\langle a,b\rangle$ of $H$ satisfies $b=[b,b^{a}]$. By von Dyck's Theorem, $K$ is an epimorphic image  of Baumslag's group $G=\langle A,B\rangle$ with  $A,B\in G$ mapped to $a,b\in K$, respectively. If $H$ is finite, then so is $K$, which forces $b=1$; so $r_4=b=1=r_1$, contradicting the injectivity of $\phi_1$. We conclude that $P$ cannot be embedded in a finite group.  

It remains to prove that $P$ embeds into the infinite Baumslag group $G$. Since the group defined by $P$ is isomorphic to $G$, we simply choose the row and column labels (and symbols) as in \eqref{eqBS2}, replacing $a$ and $b$ by $A$ and $B$, respectively. Since $G$ is non-cyclic by \cite{Bau69}, it is an easy exercise to verify that  the row labels are pairwise distinct, and similarly for the column labels and the symbols. For example, if $r_3=r_5$, then $B^{-1}=BA^{-1}B$, and so $B^{-3}=A^{-1}$, showing $\langle A,B\rangle$ is cyclic. The only nontrivial cases to consider are $r_4=r_6$, $r_5=r_6$, $c_2=c_5$ and $b=d$,  since for these it is necessary to use the defining relation $B=[B,B^A]$ as well. For example, it follows from $r_5=r_6$ that  $BA^{-1}B=B^{-1}A^{-1}$; substituting this in $B=[B,B^A]=B^{-1}A^{-1}B^{-1}ABA^{-1}BA$, one obtains $B=B^{-1}A^{-1}B^{-1}A B^{-1}$. This yields $B^A=B^{-3}$, and now it follows from $1=[B,B^A]=B$ that $\langle A,B\rangle$ is cyclic, again yielding a contradiction.
\end{proof}

\begin{figure}
\[
\begin{array}{|c||c|c|c|c|c|}\hline
& 1 & c_2&c_3&c_4&c_5\\\hline\hline
1& & b & & c &   \\\hline
r_2&a&   & c & & \\\hline
r_3& & a & b & & \\\hline
r_4&b&  & &  &d \\\hline
r_5&c&d   & &   &\\\hline
r_6& &   & & b& c\\\hline
  \end{array}
\]%
\caption{\label{figBS2}The embedding of the PLS $P$ in \lref{lemBS2}.}  
\end{figure}

\section{Computational methods} 

To guarantee that we have found a PLS of smallest possible size which embeds into an infinite group, but in no finite group, we have to show that every PLS of smaller size either embeds in some finite group, or does not embed in any group. In this section we describe some computational methods which help to test these properties. 

Throughout this section, let $P$ be an $m\times n$ PLS, with associated  presentation $\{X\mid\mathcal{R}\}$ and group $\grp{P}=\langle X\mid\mathcal{R}\rangle$.  If $P$ can be embedded in some group $H$, then $H$ is a quotient of $\grp{P}$, and $P$ also embeds in $\grp{P}$, see \lref{lemQ}. Therefore our aim is to find an embedding of $P$ in some finite quotient of $\grp{P}$, or to show that $P$ cannot be embedded in any group. The latter happens, for example, if $\grp{P}$ is the trivial group. As mentioned in the introduction, it is in general an algorithmically undecidable  problem to determine whether a finitely presented group is trivial or finite. Below we comment on our approaches, which avoided this difficulty.

\subsection{Proving that a PLS cannot be embedded into any group}\label{secNE} 

Using the relations $\mathcal{R}$ of $\grp{P}$, we can reduce the number of generators of $\grp{P}$. More precisely, we use Tietze transformation (see \cite[\S2.4.4]{handbook}) to obtain a second presentation $\{\hat{X}\mid\hat{\mathcal{R}}\}$ of $\grp{P}$, which usually has significantly fewer generators and relations. We call this presentation a \myem{reduced presentation} defined by $P$. In  the examples we have considered, that is, PLS of size at most 12, there are usually one to three generators in  $\hat{X}$ and one to three relations in $\hat{\mathcal{R}}$. We can use GAP to do these Tietze transformations, to construct an explicit isomorphism
\[\varphi\colon \langle \hat{X}\mid\hat{\mathcal{R}}\rangle\to \grp{P},\]
and to write the original generators  
\[X=\{R_i,C_j: (i,j)\in\shape(P)\}\cup\Lambda(P)\] 
as words in the new generators $\hat{X}$. After this rewriting, it often  becomes  visible that there are duplicates in the list $(R_1,\ldots,R_m)$, in $(C_1,\ldots,C_n)$, or in $\Lambda(P)$, each considered as a list of elements in the free group generated by $\hat{X}$. If this is the case, then we have established that $P$ cannot be  \emph{embedded} into $\grp{P}$. Since any embedding of $P$ into some group $H$ implies that $H$ is a quotient of $\grp{P}$, this proves that $P$ cannot be embedded in \emph{any} group. \egref{ex1}  illustrates this with an explicit PLS.
 
\begin{example}\label{ex1}
The PLS 
\[
P=\begin{array}{|c|c|c|c|}\hline
a&d & & \\\hline
 & a & d & \\\hline
& b & & c \\\hline
c& & b&a\\\hline
  \end{array}
\]yields a presentation $\{X\mid\mathcal{R}\}$ where $X=\{R_2,R_3,R_4,C_2,C_3,C_4,a,b,c,d\}$ and
\[\mathcal{R}=\{a,C_2d^{-1},R_2C_2a^{-1},R_2C_3d^{-1},R_3C_2b^{-1}, R_3C_4c^{-1}, R_4c^{-1},R_4C_3b^{-1},R_4C_4a^{-1}\}.
\]
Note that we can assume $R_1=C_1=1$ by \lref{lemrc1}. We construct $G=\grp{P}$ with the GAP code in \fref{figgap}. Applying Tietze transformations reveals that  $G\cong\mathbb{Z}$ is a one-generator group with no relations. We then express $X$ in terms of the new generator and obtain that $c=d$ in $G$. This shows that $P$ cannot be embedded in any group: in any such embedding, we would have $c=d$, a contradiction.
 
\begin{figure}[ht]
{\footnotesize
\begin{verbatim}
##############################################################################
#### first construct G=<P>
##############################################################################
gap> F:=FreeGroup(["r2","r3","r4","c2","c3","c4","a","b","c","d"]);;
gap> AssignGeneratorVariables(F);;
gap> R:=[a^-1,c2*d^-1,r2*c2*a^-1,r2*c3*d^-1,r3*c2*b^-1,
>        r3*c4*c^-1,r4*c^-1,r4*c3*b^-1,r4*c4*a^-1];;
gap> G := F/R;;
##############################################################################
#### Tietze transformations show that <P> is the infinite cyclic group
##############################################################################
gap> P:=PresentationFpGroup(G);;  
gap> TzInitGeneratorImages(P);;
gap> TzGoGo(P);
#I  there are 1 generator and 0 relators of total length 0
##############################################################################
#### now set-up homomorphisms to write original generators as a power of the
#### single new generator: we see that the generators c and d are the same
##############################################################################
gap> hom := GroupHomomorphismByImages(F,Group(GeneratorsOfPresentation(P)),
>                   GeneratorsOfGroup(F),TzImagesOldGens(P));;
gap> List([c2,c3,c4,r2,r3,r4,a,b,c,d],x->x^hom);
[ r4, r4^2, r4^-1, r4^-1, r4^2, r4, <identity ...>, r4^3, r4, r4 ] 
# thus, in G we have c=d
\end{verbatim}
}
\caption{\label{figgap}GAP code for \egref{ex1}.}
\end{figure}
\end{example}

What we have described is a very fast test to identify some PLS which cannot be embedded into any group: this is possible because our check for duplicates (in the lists of symbols, row labels, and column labels) is in the free group generated by $\hat{X}$, where the word problem is trivial. 

If this test does not establish that we cannot embed the given PLS in any group (and our attempts to find an embedding fail, cf.\ the approaches described below), then we proceed as follows. Again, let $\hat{X}$ and $\hat{\mathcal{R}}$ be as above. We use the Knuth-Bendix completion algorithm (provided by the GAP package KBMAG, see also \cite[\S12.2]{handbook}) to construct from $\hat{\mathcal{R}}$ a \emph{rewriting system} which can be used to rewrite an element of the group $\grp{P}=\langle \hat{X}\mid\hat{\mathcal{R}}\rangle$ in so-called \emph{reduced form}. If this rewriting system happens to be \emph{confluent}, then two elements in $\grp{P}$ are identical if and only if their reduced forms coincide. If the rewriting system is not confluent, then it can still happen that two different reduced forms correspond to the same element. In all our examples, this method was sufficient to prove that a PLS which we failed to embed into some finite group can in fact not be embedded into any group. We illustrate this in \egref{ex2}.

\begin{example}\label{ex2}
Starting with the PLS
\[
P=\begin{array}{|c|c|c|c|}\hline
a&b&c&\\\hline
b&d&&c\\\hline
c&&d&\\\hline
&e&&a\\\hline
 \end{array},
\]
GAP finds a reduced presentation $\{ u,v \mid v^2 u^{-2}\}$ for $\grp{P}$ with row labels, column labels, and symbols as given in \fref{figgap2}. Considering the elements in these lists in the free group on $\{u,v\}$, GAP cannot detect duplicates. However, after applying the Knuth-Bendix algorithm, we identify duplicates among the lists of reduced forms; this proves that $P$ cannot be embedded in any group. 
 
\begin{figure}[ht]
{\footnotesize
\begin{verbatim}
##############################################################################
#### start with a reduced presentation of G=<P>
##############################################################################
gap> F := FreeGroup(["u","v"]);;
gap> AssignGeneratorVariables(F);;
gap> G := F/[v^2*u^-2];;    
##############################################################################
#### the row labels, column labels, and symbols wrt this presentation;
#### considered as elements in F, these lists contain no duplicates;
###  however, note that v^-1*u^2 = v in G
##############################################################################
gap> col := [ u^0, u, v, v^-1*u ];;   row := [ u^0, u, v, u^-1*v ];;
gap> sym := [ u^0, u, v, u^2, v^-1*u^2 ];;
##############################################################################
#### using Knuth-Bendix and reduced forms to spot the duplicates
##############################################################################
gap> rws := KBMAGRewritingSystem(G);;
gap> OR  := OptionsRecordOfKBMAGRewritingSystem(rws);;    
gap> OR.maxeqns := 2000;;    
gap> KnuthBendix(rws);;  
gap> FF  := FreeStructureOfRewritingSystem(rws);;
gap> iso := GroupHomomorphismByImages(F,FF,GeneratorsOfGroup(F),
                                           GeneratorsOfGroup(FF));;
gap> symn:= List(sym,x->ReducedForm(rws,Image(iso,x)));
#WARNING: system is not confluent, so reductions may not be to normal form.
[ <identity ...>, u, v, u^2, v ]
gap> IsDuplicateFreeList(symn);
false
\end{verbatim}
}%
\caption{\label{figgap2}GAP code for \egref{ex2}.}
\end{figure}
\end{example}

\subsection{Proving that a PLS can be embedded into an abelian group}\label{secAb}

Suppose that $P$ can be embedded into an abelian group $H$. Since $H$ is a quotient of $\grp{P}$, this embedding lifts to an embedding $\phi\colon P\hookrightarrow A$ into the abelianisation $A=\grp{P}/\grp{P}'$ of $\grp{P}=\langle X\mid\mathcal{R}\rangle$; note that $A=\langle X\mid \mathcal{R}\cup\{[a,b]: a,b\in X\}\rangle$. On the other hand, we have seen in \lref{lemInfAb} that any such embedding gives rise to an embedding of $P$ into a finite abelian group. It is straightforward to check whether $P$ can be embedded in $A$: if this is the case, then $P$ can be embedded into a finite abelian group; if not, then we have proved that $P$ cannot be embedded in any abelian group.

\subsection{Finding an embedding into a nonabelian finite group}\label{secBF}

Let $P$ be as before, and suppose the method described in \sref{secAb} shows that $P$ cannot be embedded into an abelian group. Also, suppose that the method in \sref{secNE} did not conclude that $P$ fails to embed in any group. We now want to check whether $P$ can be embedded into a finite nonabelian group. Recall that if $\phi\colon P\hookrightarrow H$ is such an embedding, then $H$ is a finite quotient of $\grp{P}$. As before, we first apply Tietze transformations to find a shorter presentation of $\grp{P}$, say $\grp{P}\cong \langle \hat{X}\mid \hat{\mathcal{R}}\rangle$. We then use several \emph{quotient algorithms} and brute-force techniques to look for an embedding.\smallskip

{\bf Brute-force techniques I.} Using the small generating set $\hat{X}$, we can simply enumerate all possible homomorphisms from $\grp{P}$ into some small finite group $H$, say of order at most 24. We check if any of these homomorphisms yields an embedding of $P$.\smallskip

{\bf Nilpotent quotients.} We use the nilpotent quotient algorithm in GAP (provided by the GAP package NQ, see also \cite[\S9.4.3]{handbook}) to compute the largest quotients of $\grp{P}$ having nilpotency class $c=1,2,3,\ldots$. (For $c=1$ this quotient is the abelianisation of $\grp{P}$.) If $P$ embeds in one of these quotients, say $K_c=\grp{P}/\gamma_{c+1}(\grp{P})$, then we try to embed $P$ into a finite quotient of $K_c$, by adding several random relations. If $P$ does not embed into $K_c$, then $P$ does not embed in any nilpotent group of class $c$. Note that one can very efficiently compute with finitely presented nilpotent groups.\smallskip

{\bf Low Index Subgroups.} Based on the work in \cite{AH}, there exist efficient methods to compute with finite-index subgroups of finitely presented groups. For example, all subgroups of $\grp{P}$ of index at most $i$, where $i$ is some small integer, say $i\leq 6$, can be computed, up to conjugacy, with the GAP command {\tt LowIndexSubgroupsFpGroup}. As explained in \cite{AH}, these subgroups are stored efficiently so that one can compute their intersection $U$, and a permutation representation of $\grp{P}$ on the cosets of $U$. This yields a homomorphism from $\grp{P}$ into the symmetric group of degree $[\grp{P}:U]$, the index of $U$ in $\grp{P}$. We can readily check whether $\grp{P}$ embeds in this finite symmetric group.\smallskip

We note that, in our examples, the previous three techniques (and the methods described in \sref{secNE} and \sref{secAb}) were sufficient to deal with all occurring PLS. In practice, the Low Index Subgroup method was only required to deal with certain PLS of size 12. Below we list one other approach which could be used to consider other examples.\smallskip

{\bf Brute-force techniques II.} 
If the number of generators in $\hat{X}$ is larger than the number of relations in $\hat{\mathcal{R}}$, then $\grp{P}$ is infinite. In this case, we add some random relations to $\hat{\mathcal{R}}$ until the number of relations is at least the number of generators. We then try to compute the order of the new group using coset-enumeration techniques (see \cite[\S5]{handbook}). If the number of cosets used becomes too big, we abort this computation. (Unfortunately, the number of cosets used is not immediately related to the actual order of the group.) If we determine that the group is finite, then we check whether we have found an embedding of $P$.

\section{The candidates}\label{s:candidates}

Our mission is to find all of the smallest PLS that embed in an
infinite group, but do not embed in any finite group. On the basis of
the example given in \sref{s:eg12}, we only need to consider PLS of
size at most $12$.  Also, as mentioned in the introduction, it
suffices to consider one representative from each species of
PLS. Nevertheless, there are a great many species. 
So, in this section, we discuss techniques which allow us to
prune the list of species that we have to consider to a manageable number.


In this section, $\Pi_i$ denotes the 
projection of a triple onto its $i$-th coordinate, and $\Pi_{i,j}$ denotes 
the projection onto coordinates $i,j$.
For every PLS $P$ there is an associated graph formed by taking the
triples $\trip(P)$ of $P$ as vertices and joining two vertices by an
edge if the corresponding triples agree in any coordinate. We say that
$P$ is \myem{connected} if the associated graph is connected,
otherwise it is \myem{disconnected}. For our purposes, we only need to
consider connected PLS, because of our next result. For this result
only, we need to broaden our definition of embedding to allow the row
and column indices of a PLS to be any set of positive integers, rather
than insisting that they are consecutive starting at 1. However, this
is merely a technicality, which can be overcome by appropriate relabelling.

\begin{lemma}\label{l:connect}
Let $P,P_1,P_2$ be PLS for which
$\trip(P)=T_1\cup T_2$, where $T_i=\trip(P_i)$ for $i=1,2$ and
$\Pi_i(T_1)\cap\Pi_i(T_2)=\emptyset$ for $i=1,2,3$.
Suppose that there are embeddings
$P_1\hookrightarrow G$ and $P_2\hookrightarrow H$
for groups $G,H$. Then $P$ embeds into $G\times H\times C$, 
where $C$ is a cyclic group of order $3$.
\end{lemma}

\begin{proof}
Let $g$ denote a generator for $C$. Suppose that we have embeddings
$(\phi_1,\phi_2,\phi_3):P_1\hookrightarrow G$ and 
$(\phi'_1,\phi'_2,\phi'_3):P_2\hookrightarrow H$. It is routine to
verify that the three maps
\begin{align*}
x&\mapsto
\begin{cases}
(\phi_1(x),1_H,1_C)&x\in\Pi_1(T_1),\\
(1_G,\phi'_1(x),g)&x\in\Pi_1(T_2),\\
\end{cases}
\\
y&\mapsto
\begin{cases}
(\phi_2(y),1_H,1_C)&y\in\Pi_2(T_1),\\
(1_G,\phi'_2(y),g)&y\in\Pi_2(T_2),\\
\end{cases}
\\
z&\mapsto
\begin{cases}
(\phi_3(z),1_H,1_C)&z\in\Pi_3(T_1),\\
(1_G,\phi'_3(z),g^2)&z\in\Pi_3(T_2),\\
\end{cases}
\end{align*}
are injections that provide an embedding $P\hookrightarrow G\times H\times C$.
\end{proof}

In particular, \lref{l:connect} shows that if $P_1$ and $P_2$ both embed in
finite groups, then so does their disconnected union $P$. Also, if
either $P_1$ or $P_2$ cannot be embedded in an infinite group, then
it is immediate that $P$ cannot either. These two observations combine
to tell us that the smallest PLS that embeds in an infinite group but
not in any finite group will be connected.
As an aside, we note that the disconnected PLS in \fref{figNotC6} does
not embed in the cyclic group of order 6, showing that the cyclic
factor $C$ in \lref{l:connect} cannot be abandoned.

\begin{figure}[ht]
{
\[\begin{array}{|c|c|c|c|c|}\hline
a & b & & & \\\hline
b & a & & & \\\hline
  &   & c & d & e \\\hline
  &   & d & e & c \\\hline
  &   & e & c & d \\\hline
\end{array}
\]}%
\caption{\label{figNotC6}A disconnected PLS, 
which does not embed into the cyclic group of order 6.}
\end{figure}  


In \cite{WW17}, various restrictions were shown in terms of which PLS
need to be considered when studying embeddings in groups. The next
lemma is a direct adaptation of \cite[Lem.10\&11]{WW17}, which helps
us to rule out many PLS.

\begin{lemma}\label{l:rulenoopt}
Let $P$ be a PLS and
$T=\{(i,j,P_{i,j}):(i,j)\in\shape(P)\}$. Suppose that either
\begin{enumerate}
\item[\rm 1)]
there exists $(r,c,s)\in T$ such that 
$T'=T\setminus\{(r,c,s)\}$ satisfies
$r\notin\Pi_1(T')$ and for each $(r',c',s')\in T'$ 
either $(r',c)\in\Pi_{1,2}(T')$ or $(r',s)\in\Pi_{1,3}(T')$, or
\item[\rm 2)] there exists $\ell\ge1$ and distinct
$(r,c_1,s_1),\dots,(r,c_\ell,s_\ell)\in T$ for which
$T'=T\setminus\{(r,c_i,s_i):1\le i\le\ell\}$ satisfies $r\notin\Pi_1(T')$
and $$\{1\le i\le\ell:c_i\in\Pi_2(T')\}\cap\{1\le i\le\ell:s_i\in\Pi_3(T')\}=\emptyset.$$
\end{enumerate}
Then the smallest PLS to embed in an infinite group, but not in any
finite group, cannot be in the same species as $P$.
\end{lemma}

For ease of exposition, \lref{l:rulenoopt} has not been stated in a
form that displays the symmetry between the three coordinates in
$T$. Hence when using it to check whether we need to consider a
particular PLS, it is important to consider all three PLS obtained by
uniformly cyclically permuting the triples in~$T$. We call any
connected PLS that cannot be ruled out using \lref{l:rulenoopt} a
\myem{candidate}. \Tref{tab:nrcand} shows, for each size up to 12, the
number of all species of PLS, species of connected PLS, and species of
candidate PLS. 
The catalogue of PLS for each size $s$ was generated by extending each
species representative of size $s-1$ in all possible different ways,
then using \texttt{nauty}~\cite{nauty} to screen for species
representatives.
The number of species of PLS for sizes at most $8$ has
been independently confirmed by F\'alcon and Stones~\cite{FS17}.  The
number of species of PLS for sizes at most $7$ were previously
reported in \cite{WW17}.

\begin{table}
\begin{center}
{\small\noindent
\begin{tabular}{l||r|r|r|r|r|r|r|r|r|r|r|r}
size&1&2&3&4&5&6&7&8&9&10&11&12\\
\hline\hline
all&1&2&5&18&59&306&1861&15097&146893&1693416&22239872&327670703\\
conn.&1&1&3&11&36&213&1405&12274&125235&1490851&20003121&299274006\\
cand.&0&0&0&2&0&11&50&489&6057&92533&1517293&27056665\\
\end{tabular}}
\caption{\label{tab:nrcand}Number of species of candidate PLS of size at most 12.}
\end{center}
\end{table}

\section{The outcome}

In this section we report the results of our computations, which employed
the techniques developed in the previous sections. 

For each candidate PLS $P$ as identified in \sref{s:candidates}, we
proceeded as follows. First, we computed $\grp{P}$ and a reduced
presentation $\grp{P}\cong \langle
\hat{X}\mid\hat{\mathcal{R}}\rangle$, see \sref{secNE}. If we then found
duplicates in the list of row labels, column labels, or symbols,
then we knew that $P$ could not be embedded in any group, so we stopped
considering it. If we found no duplicates, then we proceeded with further
tests. If $\hat{\mathcal{R}}=\emptyset$,
then $\grp{P}$ was a free group and $P$ embedded into a finite group
(\cyref{cy:lemRF}); we then checked whether $P$ also embedded in a
finite abelian group (\lref{lemInfAb}). Similarly, if
$\hat{\mathcal{R}}\ne\emptyset$, then we first checked whether $P$
embedded in a finite abelian group. If not, then we applied Brute Force
Techniques I and the Nilpotent Quotients method (\sref{secBF}) to find
an embedding of $P$ into some finite nonabelian group. If we did not
succeed, then we applied the Knuth-Bendix completion algorithm (which
might tell us that $P$ cannot be embedded in any group, see
\sref{secNE}) and the Low Index Subgroup method (which might find an
embedding into a finite group, see \sref{secBF}). After all these
tests, there were only 50 species of PLS remaining for which
we did not reach a conclusion. All had size 12. 
Below we describe how we proved that
these PLS can be embedded into an infinite group, but in no finite
group. First we summarise in \Tref{tab1} the results of our
computations for each candidate of size at most
$12$. The headings on the columns of \Tref{tab1} indicate
which family of PLS is being counted, according to the following schema:

\medskip

\begin{tabular}{lcp{11cm}}
NE&:&  cannot be embedded in any group;\\
abelian&:&  can be embedded in a finite abelian group;\\
nonabelian&:&  can be embedded in a finite nonabelian group, but not in any abelian group;\\
infNotFin &:&  can be embedded in an infinite group, but not in any finite group.
\end{tabular}

\begin{table}[h]
{\small
\begin{center} 
\begin{tabular}{r||r|r|r|r}
size &  NE &            abelian &  nonabelian &   infNotFin\\\hline\hline
4    &     0  &        2    &        0     &   0  \\
6    &     0  &        10    &        1    &   0   \\
7    &     2  &       44    &        4     &  0  \\
8    &    16  &      435    &       38     & 0   \\
9    &   147  &     5447    &      463     &  0   \\
10   &  2402  &    82555    &     7576     &  0   \\
11   & 42884&    1338816    &   135593     & 0 \\
12   &854559 &  23520406    &  2681650     &    50     \\\hline 
\end{tabular}
\caption{\label{tab1}Counts of species of candidate PLS 
which embed in certain groups}
\end{center}
}
\end{table}

We comment on the 50 PLS of size 12 which can be embedded into an infinite group, but in no finite group. In all but 6 cases, the reduced presentation of $P$ is $\{a,b \mid b=[b,b^a]\}$, hence $\grp{P}$ is isomorphic to the Baumslag group $G=\langle a,b \mid b=[b,b^a]\rangle$ as defined in \lref{lemBS2}. In the remaining 6 cases, the reduced presentations show that $\grp{P}$ is isomorphic to
\[B_1=\langle a,b\mid b=[b,(b^{-2})^a]\rangle\quad\text{or}\quad B_2=\langle a,b\mid b=[b,(b^{2})^a]\rangle,\] respectively. In \fref{figB1B2} we present two PLS $P_1$ and $P_2$ with corresponding group $B_1$ and $B_2$, respectively.
In  $B_1$ we have $b^{-1}=[b^{-1},(b^{-2})^a]$, and we define  $H_1=\langle b^{-2},a\rangle$. In $B_2$ we observe  $[b^2,(b^2)^a]=[b,(b^2)^a]^b[b,(b^2)^a]=b^bb=b^2$, and we define $H_2=\langle b^2,a\rangle$. It is easy to verify by computer (using coset-enumeration) that $H_1=B_1$ and $H_2=B_2$; moreover, the function {\tt SearchForIsomorphism} provided by the computer algebra system Magma \cite{magma} proves that $G\to H_1$, $(a,b)\mapsto (a,b^{-2})$, and  $G\to H_2$, $(a,b)\mapsto (a,b^{2})$, are isomorphisms. In conclusion, $G\cong B_1\cong H_1\cong B_2\cong H_2$, and $b^{-2}$ and $b^2$ are trivial in every finite quotient of $H_1$ and $H_2$, respectively.

\begin{figure}[ht]
{
\[
P_1= \begin{array}{|c|c|c|c|c|c|}\hline
a&b&c&d&& \\\hline
b&e& & &c& \\\hline
&&e&&a&d\\\hline
&&&e&&a\\\hline
\end{array}
\hspace*{2cm}
P_2=\begin{array}{|c|c|c|c|c|}\hline
a&b&c& & \\\hline
b&d&&c&\\\hline
e&&&&d\\\hline
&&d&a& \\\hline
&&&e&a\\\hline
  \end{array}
\]}%
\caption{\label{figB1B2}Two PLS with corresponding groups $B_1$ and $B_2$, 
respectively.}
\end{figure} 
 
Each of the 50 PLS embeds into $\grp{P}$. This can be proved as in the proof of \lref{lemBS2} by showing that there are no duplicates among the list of symbols, row labels, and columns labels, respectively.  In our case, we showed by computer that if there were duplicates, then the resulting group $\grp{P}$ was cyclic, which we knew was not the case. This established that each of the 50 PLS embeds into $\grp{P}$. In each case, the structure of the embedded PLS implied that no embedding into a finite group exists. For this deduction we made frequent use of the facts that $b=1$ in any finite quotient of $G$ and $b^2=1$ in any finite quotient of $H_1$ or $H_2$. We thereby established \tref{t:main}. 
 
We conclude with two remarks. First, it is known that the word problem
in one-relator groups is solvable, but also that counter-intuitive
things can happen. For example, it is shown in \cite{mccool} that for
every integer $t\geq 3$, the group $\langle a,b\mid ba^t b^2
a^t\rangle$ satisfies $[b^3,a]=1$, even though the defining relator
has length $2t+3$. This indicates that one has to be careful when
checking for embeddings into one-relator groups such as the Baumslag
group. Words much shorter than the defining relator may, in principle,
equate to the identity.
Second, recall that Hirsch and Jackson found a PLS of size 29
which can be embedded in an infinite group but in no finite group. As
shown in \cite[Ex.3.7]{HJ12}, their PLS embeds into the group 
$\langle a,b,c,d\mid ab=b^2a,\; bc=c^2b,\; cd=d^2c,\; da=ad^2\rangle$ 
defined by Higman, who also proved that this group has no nontrivial
finite quotients. Thus the example in \cite{HJ12} works for a different 
underlying reason than the 50 examples mentioned in \tref{t:main}. 
Data from this project, including those 50 examples, 
can be downloaded from \cite{WWWW}.

  \let\oldthebibliography=\thebibliography
  \let\endoldthebibliography=\endthebibliography
  \renewenvironment{thebibliography}[1]{%
    \begin{oldthebibliography}{#1}%
      \setlength{\parskip}{0.4ex plus 0.1ex minus 0.1ex}%
      \setlength{\itemsep}{0.4ex plus 0.1ex minus 0.1ex}%
  }%
  {%
    \end{oldthebibliography}%
  }

\section*{References}

\bibliographystyle{plain}

\end{document}